\documentclass[12pt]{amsart}
\usepackage[english]{babel}
\usepackage{graphicx}
\usepackage{amsfonts}
\usepackage{amsmath}
\usepackage{amssymb}
\usepackage{amscd}
\usepackage[all,cmtip,matrix, arrow]{xy}
\usepackage{amsthm}
\def\Aut {\operatorname{Aut}}

\newcommand{\PP}{\mathbb{P}}
\newcommand{\QQ}{\mathbb{Q}}
\newcommand{\Pic}{\operatorname{Pic}}
\newcommand{\Stab}{\operatorname{Stab}}

\newtheorem{lemma}{Lemma}[section]
\newtheorem{theorem}[lemma]{Theorem}
\newtheorem{propos}[lemma]{Proposition}

\newtheorem{corollary}[lemma]{Corollary}

\theoremstyle{definition}
\newtheorem{definition}[lemma]{Definition}

\newtheorem{remark}[lemma]{Remark}
\makeatletter
\def\blfootnote{\xdef\@thefnmark{}\@footnotetext}
\makeatother
\begin{document}

\title{Biregular and birational geometry of quartic double solids with 15 nodes}
\author[A.\,A.~Avilov]{A.\,A.~Avilov}
\address{National Research University Higher School of Economics, AG Laboratory, HSE, 6 Usacheva str., Moscow, Russia, 119048.}
\email{v07ulias@gmail.com} 

\maketitle

\begin{abstract}

Three-dimensional del Pezzo varieties of degree 2 are double covers of projective space $\mathbb{P}^{3}$ branced in a quadric. In this paper we prove that if a del Pezzo variety of degree 2 has exactly 15 nodes then the corresponding quadric is a hyperplane section of the Igusa quartic or, equivalently, all such del Pezzo varieties are members of one particular linear system on the Coble fourfold. Their automorphism groups are induced from the automorphism group of Coble fourfold. Also we classify all $G$-birationally rigid varieties of such type.

Bibliography: 11 titles.
\end{abstract}

\markright{Quartic double solids with 15 nodes}

\blfootnote{I was partially supported by the Russian Academic Excellence Project '5-100' and by ``Young Russian Mathematics'' award.}

\section{Introduction} In this article we work over the field of complex numbers.

Classification of rational $G$-Fano varieties is an important problem for classification of finite subgroups of Cremona groups (cf.~\cite{DI1}). Three-dimensinal $G$-del Pezzo varieties were partially classified by Yu. Prokhorov in~\cite{Pro1}. In the end of the article he posed the question: which $G$-del Pezzo threefolds are birationally (super)rigid? Cases of rational $G$-del Pezzo threefolds of degree 3 and 4 were considered by the author in papers~\cite{Avi1} and~\cite{Avi2}. In the case of degree 2 the situation is much more complicated --- it can be seen already from the classification of rational del Pezzo threefolds of degree 2 (see~\cite{CPS3}). Such varieties have at most 16 singularities (see, for example,~\cite{Pro1}). In the present paper we consider the case of del Pezzo threefolds of degree 2 with 15 ordinary double points. The main results of this article are the following theorems:
\begin{theorem}\label{th57} Every quartic surface with $15$ nodes is a hyperplane section of the Igusa quartic. Let $X$ be a del Pezzo threefold of degree 2 with precisely 15 ordinary double points. Then the variety $X$ is a member of a linear system $|\mathcal{L}|$ on the Coble fourfold (i.e. the cover of $\PP^{4}$ branched in the Igusa quartic), where $\mathcal{L}$ is a restriction of the line bundle $\mathcal{O}(1)$ on the weighted projective space $\PP(2, 1, 1, 1, 1, 1)$, and the automorphism group of $X$ coincides with the stabilizer of $X$ in the automorphism group of the Coble fourfold.
\end{theorem}
\begin{theorem}\label{th2} Let $X$ be a del Pezzo threefold of degree 2 with 15 ordinary double points. Let~$G\subset \Aut(X)$ be a subgroup such that the variety $X$ is $G\QQ$-factorial and $\operatorname{rk}\Pic(X)^{G}=1$. Then the variety $X$ is $G$-birationally rigid only in the following situation: it can be given by the equation $$y^{2}-4\sum\limits_{i=1}^{5}x_{i}^{4}+(\sum\limits_{i=1}^{5}x_{i}^{2})^{2}=\sum\limits_{i=1}^{5}x_{i}=0$$ in $\PP(2, 1, 1, 1, 1, 1)$ and $G$ is isomorphic to $S_{5}\times C_{2}$, $A_{5}\times C_{2}$ or $S_{5}$ (non-standard subgroup). Moreover, in this case $X$ is $G$-birationally superrigid.
\end{theorem}
In this article we use the following notation for groups: by $C_{n}$ we denote the cyclic group of order $n$; by $D_{2n}$ we denote the dihedral group of order $2n$; by $S_{n}$ we denote the symmetric group of rank~$n$; by $A_{n}$ we denote the alternating group of rank $n$.

The author would like to thank I. Cheltsov, Yu. Prokhorov and C. Shramov for useful discussions and comments. The author also would like to thank the organizers of the conference ``Subgroups of Cremona groups'' in Oberwolfach in June 2018 where this article was written.

\section{Biregular geometry of del Pezzo threefolds of degree 2 with 15 nodes}
  Let $X$ be a double cover of $\PP^{3}$ branched in a quartic which has 15 ordinary double points and no other singularities or, equivalently, a del Pezzo threefold of degree $2$ with 15 ordinary double points. In the sequel we will call them quartic double solids with 15 nodes. There is the following statement:
\begin{propos}[{\cite[Theorem 8.1]{Pro1}}]\label{pr1} The variety $X$ can be obtained by the following diagram:

\[\xymatrix{
 & \widehat{X}\ar[rd]^{\pi}\ar[ld]_{\tau} & &
\\
X & & H\ar@{^(->}[r] & \PP^{2}\times\PP^{2}
}
\]
  where $H$ is a smooth divisor of bi-degree $(1, 1)$ in $\PP^{2}\times\PP^{2}$, the morphism $\pi$ is a blow up of four points $P_{i}\in H$ in general position and $\tau$ is a contraction of fifteen rational curves which has zero intersection number with the canonical class $K_{\widehat{X}}$. More presicely, the proper transforms of the following curves are contracted:
\begin{itemize}
\item curves of bi-degree $(1,0)$ and $(0,1)$ passing through $P_{i}$;
\item curves of bi-degree $(1,1)$ passing through a pair of points $P_{i}$;
\item curve of bi-degree $(2,2)$ passing through all the points $P_{i}$.
\end{itemize}
\end{propos}
\begin{definition} \emph{The Igusa quartic} $I$ is the three-dimensional quartic which can be explicitly given by the following system of equations in $\PP^{5}$:
$$s_{1}=4s_{4}-s_{2}^{2}=0,$$ 
where $s_{j}=\sum\limits_{i=1}^{6}x_{i}^{j}$ (for our purposes such numeration of coordinates in $\PP^{5}$ is more convenient). \emph{The Coble fourfold} $Z$ is a double cover of $\PP^{4}$ ramified in the Igusa quartic $I$. In other words, $Z$ can be explicitly given by the following system of equations in $\PP(2, 1, 1, 1, 1, 1, 1)$ with weighted homogeneous coordinates $z, x_{1},\ldots, x_{6}$:
$$s_{1}=0,\ z^{2}=4s_{4}-s_{2}^{2}.$$

\end{definition}
The following structure theorem was proved recently by I. Cheltsov, A. Kuznetsov and C.~Shramov:
\begin{theorem}[{\cite[Theorem 1.9, Proposition 2.21]{CKS1}}]\label{th1} The Coble fourfold $Z$ can be obtained by the following diagram:
$$\xymatrix{
 & \widehat{Z}\ar[rd]^{\xi}\ar[ld]_{\phi} &
\\
Z & & \PP^{2}\times\PP^{2}
}$$
  where $\xi$ is a blow up of four points $P_{i}=(p_{1, i}, p_{2, i})$ in general position (without loss of generality we can assume that such points have coordinates $$(1:0:0,1:0:0),\ (0:1:0,0:1:0),\ (0:0:1,0:0:1),\ (1:1:1,1:1:1)$$ respectively). The map $\phi$ is a small morphism which contracts the proper transform of the following fifteen surfaces:
\begin{itemize}
\item eight planes of the form $p_{1, i}\times\PP^{2}$ or $\PP^{2}\times p_{2, i}$;
\item six quadrics of the form $l_{i, j}\times l'_{i,j}$ where $l_{i, j}$ (resp. $l'_{i, j}$) is a line in $\PP^{2}$ passing through the points $p_{1, i}$ and $p_{1, j}$ (resp. $p_{2, i}$ and $p_{2, j}$);
\item diagonal in $\PP^{2}\times\PP^{2}$.
\end{itemize}
\end{theorem}

From the previous theorem and Proposition~\ref{pr1} we deduce the following proposition.
\begin{propos}\label{pr2} Let $X$ be a quartic double solid with 15 ordinary double points. Then $X$ is isomorphic to a double cover of $\PP^{3}$ branched in a hyperplane section of the Igusa quartic.
\end{propos}
\begin{proof} We will use the notation which was introduced is Proposition~\ref{pr1}. We have the following commutative diagram:
$$\xymatrix{
\widehat{X}\ar[d]^{\pi} \ar@{^(->}[r]& \widehat{Z}\ar[d]^{\xi}
\\
H\ar@{^(->}[r]& \PP^{2}\times\PP^{2}
}$$
  where $H$ is a smooth divisor of bi-degree $(1, 1)$ and the morphisms $\xi$ and $\pi$ are blow ups of four points on $H$ in general position. It is easy to see that intersections of $15$ surfaces listed in Theorem~\ref{th1} with the variety $\widehat{X}$ are precisely curves contracted by the morphism $\tau$. Hence we have the following commutative diagram:
$$\xymatrix{
 & & \widehat{X}\ar[rd]^{\pi}\ar[lld]_{\tau} \ar@{^(->}[r]& \widehat{Z}\ar[rd]^{\xi}\ar[lld]_{\phi} &
\\
X\ar@{^(->}[r]\ar[d]^{\phi'|_{X}} & Z\ar[d]^{\phi'} & & H\ar@{^(->}[r]& \PP^{2}\times\PP^{2}
\\
S\ar@{^(->}[r] & \PP^{4} & & &
}$$
Note that the divisor $\widehat{X}$ is equivalent to $-\frac{1}{3}K_{\widehat{Z}}$ and the morphism $\phi'\circ\phi$ is given by the linear system $|-\frac{1}{3}K_{\widehat{Z}}|$ (see~\cite[Section 2.1]{CKS1}), thus $S$ is a hyperplane section of the Igusa quartic $I$ and~$X$ is a subvariety of the Coble fourfould given by a linear equation in coordinates $x_{1},\ldots, x_{6}$.
\end{proof}
\begin{corollary}\label{cor3} Every quartic surface with $15$ nodes is a hyperplane section of the Igusa quartic.
\end{corollary}
\begin{remark} In the paper~\cite{RRS} the authors proved that a general quartic threefold with 15 ordinary double points is a hyperplane section of the Igusa quartic. From the previous corollary we see that in fact every quartic with 15 nodes has the same property.
\end{remark}
\begin{remark} In the sequel we will say that the variety $X$ is a hyperplane section of the Coble fourfold, although it is not completely correct, since the corresponding linear system is not very ample.
\end{remark}
For classification of automorphism groups of del Pezzo threefolds of degree 2 with 15 nodes we need the following well-known properties of the Igusa quartic.
\begin{propos}\begin{enumerate}
\item The automorphism group of the Igusa quartic is isomorphic to $S_{6}$ and acts by permutations of coordinates.
\item The singular set of the Igusa quartic $I$ consists of $15$ lines which can be described explicitly by the following equations: $$x_{\sigma(1)}=x_{\sigma(2)}, x_{\sigma(3)}=x_{\sigma(4)}, x_{\sigma(5)}=x_{\sigma(6)}$$ where $\sigma$ is an element of the group $S_{6}$. Such lines we denote by $l_{\alpha}$ where $\alpha$ is a partition of the set $\{1, 2, 3, 4, 5, 6\}$ into three pairs of elements.
\item There are exactly $10$ hyperplanes whose intersection with the Igusa quartic is a quartic surface with multiplicity $2$. They can be explicitly given by the following equations:
     $$x_{\sigma(1)}+x_{\sigma(2)}+x_{\sigma(3)}=x_{\sigma(4)}+x_{\sigma(5)}+x_{\sigma(6)}=0$$
      where $\sigma$ is an element of the group $S_{6}$. We denote such hypersurfaces by $H_{\beta}$ where $\beta$ is a partition of the set $\{1, 2, 3, 4, 5, 6\}$ into two triples of elements.
\item Every hyperplane $H_{\beta}$ contains exactly six lines $l_{\alpha}$ while every line $l_{\alpha}$ lies exactly on four hyperplanes $H_{\beta}$. In other words, they form a $(15_{4}, 10_{6})$-configuration in notation of~\cite{Dol2}. Also, every pair of hyperplanes $H_{\beta}$ contain exactly two common lines $l_{\alpha}$.
\end{enumerate}
\end{propos}
\begin{remark} These properties can be easily deduced from the fact that the Igusa quartic is the dual variety of the Segre cubic, while singular points of the Segre cubic and planes on it form a $(15_{4}, 10_{6})$-configuration (see~\cite[\S 9.4.4]{Dol1}).
\end{remark}

\begin{definition} \emph{The automorphism of the} $(15_{4}, 10_{6})$-\emph{configuration} is a permutation of sets~$H_{\beta}$ and $l_{\alpha}$ which preserves the relation ``a line lies on a plane''.
\end{definition}
\begin{lemma} The automorphism group of the $(15_{4}, 10_{6})$-configuration is isomorphic to $S_{6}$.
\end{lemma}
\begin{proof} Let $G$ be the automorphism group of the $(15_{4}, 10_{6})$-configuration. Obviously, the group $\Aut(I)\simeq S_{6}$ acts faithfully on the configuration of singular lines and hyperplanes, so $S_{6}\subset G$. Hence it is enough to prove that $|G|\leq 720$.

  Since $G$ acts transitively on the set of lines $l_{\alpha}$, we see that $$|G|=15|\Stab(l_{((1,2)(3,4)(5,6))})|.$$ The group  $\Stab(l_{((1,2)(3,4)(5,6))})$ preserves the set of four hyperplanes $H_{\beta}$ containing the li\-ne~$l_{((1,2)(3,4)(5,6))}$, so $$|G|\leqslant 15\cdot 24\cdot |\Stab(H_{((1, 3, 5)(2, 4, 6))}, H_{((1, 3, 6)(2, 4, 5))}, H_{((1, 4, 5)(2, 3, 6))}, H_{((1, 4, 6)(2, 3, 5))})|.$$
  The last group (let us denote it by $G'$) fixes also six lines which lie in pairwise intersections of four fixed hyperplanes, in particular it fixes the line $l_{((1,4)(2,3)(5,6))}$. Two remaining hyperplanes passing through the line $l_{((1,4)(2,3)(5,6))}$ form a $G'$-invariant set. One can easily check that only the trivial element of $G'$ fixes them. Hence we deduce that $$|G|\leqslant 15\cdot 24\cdot 2=720.$$
\end{proof}
\begin{propos}\label{pr3} Let $X$ be a quartic double solid with $15$ nodes. Let $X=Z\cap \bar{H}$ be a representation of $X$ as a hyperplane section of the Coble fourfold. Then the automorphism group of $X$ coincides with the stabilizer of $\bar{H}$ in the automorphism group of the Coble fourfold~$Z$.
\end{propos}
\begin{proof} Due to~\cite[Corollary 3.5]{CKS1} the automorphism group of the variety $Z$ is isomorphic to $S_{6}\times C_{2}$ where the group $S_{6}$ acts by permutations of coordinates $x_{i}$ and $C_{2}$ acts by the change of sign of the coordinate $y$. The variety $X$ is a double cover of $\PP^{3}$ branched in a hyperplane section of the Igusa quartic. Let us denote such a hyperplane by $H$. Since the double cover $X\to H$ is given by the linear system $|-\frac{1}{2}K_{X}|$, we have a homomorphism of groups $\Aut(X)\to \Aut_{lin}(I\cap H)$ where $\Aut_{lin}(I\cap H)$ is the group of linear transformations of $H$ which preserve the quartic $I\cap H$. Obviously this homomorphism is surjective and its kernel is generated by the Galois involution. Thus we need to prove that $\Aut_{lin}(I\cap H)$ coincides with the stabilizer of the hyperplane $H$ in the group $\Aut(I)\simeq S_{6}$.

 Intersections of lines $l_{\alpha}$ and hyperplanes $H_{\beta}$ with the subspace $H$ form a $(15_{4}, 10_{6})$-configuration of singular points of $I\cap H$ and planes which intersect $I\cap H$ in a double conic. Since the singular points are in a general enough position, the natural map from $\Aut_{lin}(I\cap H)$ to an automorphism group of $(15_{4}, 10_{6})$-configuration is an embedding. Since we have a natural isomorphism between the automorphism group of the $(15_{4}, 10_{6})$-configuration and the group $\Aut(I)$, we obtain a natural embedding $\Aut_{lin}(I\cap H)\hookrightarrow \Aut(I)$.

Suppose that the image of some element of the group $\Aut_{lin}(I\cap H)$ in $\Aut(I)$ does not preserve the hyperplane $H$. Then this automorphism induces a linear map from $H$ to another hyperplane $H'$ such that the point $l_{\alpha}\cap H$ maps to $\l_{\alpha}\cap H'$ for every $\alpha$. Since hyperplanes $H$ and~$H'$ don't coincide, one has $H\cap H_{\beta}\neq H'\cap H_{\beta}$ for some index $\beta$. Note that $H_{\beta}\cap I$ is a smooth quadric surface. One can easily check that among six lines $l_{\alpha}$ lying on the quadric $H_{\beta}\cap I$ three lines lie in one family while another three lines lie in another family. One can easily see that every isomorphism between two conics with six marked points which are hyperplane sections of a quadric can be uniquely extended to an automorphism of the quardic which preserves families of lines. We apply this statement to the case of the map $H\cap H_{\beta}\cap I\to H'\cap H_{\beta}\cap I$. But if an automorphism of $H_{\beta}\cap I$ preserves all lines $l_{\alpha}\subset H_{\beta}$ then such an automorphism is trivial, so~$H\cap H_{\beta}\cap I = H'\cap H_{\beta}\cap I$, which contradicts our assumptions.

This contradiction shows that the image of the embedding $\Aut_{lin}(I\cap H)\hookrightarrow \Aut(I)$ is contained in the stabilizer of the corresponding hyperplane $\Stab(H)\subset \Aut(I)$. Obviously, inverse statement also holds, so $\Aut_{lin}(I\cap H)= \Stab(H)$ and $\Aut(X)=\Stab(\bar{H})$.
\end{proof}

As a consequence of Proposition~\ref{pr2} and Proposition~\ref{pr3} and Corollary~\ref{cor3} we get Theorem~\ref{th57}.

\section{Equivariant birational rigidity of quartic double solids with 15 nodes}
\begin{definition} Let $X$ and $Y$ be a varieties with an action of a finite group $G$. We call a rational map $f:X\dasharrow Y$ a \emph{$G$-equivariant map} if there exist an automorphism $\tau$ of the group~$G$ such that the following diagram commutes for every $g\in G$:
$$\xymatrix{
X \ar@{-->}[r]^{f}\ar[d]^{g}& Y \ar[d]^{\tau(g)}
\\
X \ar@{-->}[r]^{f}& Y
}$$
We denote the group of $G$-equivariant automorphisms of a $G$-variety $X$ by $\Aut^{G}(X)$ and the group of $G$-equivariant birational selfmaps of a $G$-variety $X$ by $\operatorname{Bir}^{G}(X)$.
\end{definition}
\begin{definition} A $G$-Fano variety $X$ is called \emph{$G$-birationally rigid} if there is no $G$-Mori fibrati\-on~$X'\to Y$ such that varieties $X$ and $X'$ are $G$-birationally equivalent but not isomorphic. If one also has $\operatorname{Bir}^{G}(X)=\Aut^{G}(X)$ then $X$ is called \emph{$G$-birationally superrigid}.
\end{definition}

As an application of Theorem~\ref{th57} we classify all $G$-birationally rigid del Pezzo threefolds of degree 2 with 15 nodes.
\begin{lemma}\label{le1} Let $X=Z\cap H$ be a quartic double solid with $15$ nodes where $Z\subset\PP(2, 1, 1, 1, 1, 1, 1)$ is a Coble fourfold. Assume that $X$ is $\Aut(X)$-birationally rigid. Then the equation of $H$ is of the form $x_{i}+ax_{j}=0$ where $a\neq -1$. The group $\Aut(X)$ in this case is isomorphic to
\begin{itemize}
\item $S_{5}\times C_{2}$ if $a=0$;
\item $S_{4}\times C_{2}\times C_{2}$ if $a=1$;
\item $S_{4}\times C_{2}$ in other cases.
\end{itemize}
\end{lemma}
\begin{proof} Since $\Aut(X)$ always contains the Galois involution of the double cover, the variety~$X$ is always $\Aut(X)\QQ$-factorial (i.e. every $\Aut(X)$-invariant divisor is a $\QQ$-Cartier divisor) and $\Aut(X)$-minimal (i.e. the rank of the invariant Picard group equals to $1$). We know that $\Aut(X)\simeq G\times C_{2}$, where $G$ is a subgroup of $\Aut(I)\simeq S_{6}$. We consider the projective space~$\PP^{4}\supset I$ as a projectivization of the simplicial representation $W$ of the group $S_{6}$. Consequently, $H$ is a projectivization of a four-dimensional subrepresentation $V\subset W$ of the group~$G$. If $V$ contains a two-dimensional subrepresentation of $G$ then its projectivization is a $G$-invariant line on $\PP(V)$. Projection from this line gives us a structure of $\Aut(X)$-equivariant fibration by rational surfaces. We can apply an $\Aut(X)$-equivariant resolution of singularities and $\Aut(X)$-equivariant relative minimal model program and obtain a $G$-Mori fiber space with the base of positive dimension which is birational to our quartic double solid, a contradiction. Hence the representation $W$ of the group $G$ is a direct sum either of a one-dimensional and a four-dimensional irreducible representations or of two one-dimensional and a three-dimensional irreducible representations. Also we know that there are no $\Aut(X)$-invariant singular points on~$X$ (otherwise projection from such a point gives us a structure of an $\Aut(X)$-equivariant conic bundle).

We have only the following non-abelian subgroups of $S_{6}$ which are stabilizers of hyperplanes in $\PP(W)$: the group $S_{5}$ (for the hyperplane $x_{1}=0$), the group $S_{4}\times C_{2}$ (for the hyperplane~$x_{1}\pm x_{2}=0$), the group $S_{4}$ (for the hyperplane $x_{1}+ax_{2}=0$), the group $S_{3}\times S_{3}$ (for the hyperplane $x_{1}+x_{2}+x_{3}=0$), the group $S_{3}\times C_{3}$ (for the hyperplane $x_{1}+\xi x_{2}+\xi^{2}x_{3}=0$, where $\xi^{3}=1$), the group $S_{3}\times C_{2}$ (for the hyperplane $x_{1}+x_{2}+ax_{3}=0$), and the group $S_{3}$ (for the hyperplane $x_{1}+ax_{2}+bx_{3}=0$). For every subgroup one can calculate the character of the representation and decompose the representation $W$ in irreducible summands. It turns out that only stabilizers of hyperplanes $\{x_{i}+ax_{j}=0\}$ satisfy the properties mentioned above.

If $a=-1$ then some lines $l_{\alpha}$ lie on $X$, which is impossible since $X$ has isolated singularities.
\end{proof}

\begin{propos}\label{pr58} In the notation of the previous proposition, assume that $a\neq 0$. Then $X$ is not $\Aut(X)$-birationally rigid.
\end{propos}
\begin{proof} Denote $\Aut(X)$ by $G$ for simplicity. Without loss of generality we may assume that $X$ is given by the equation $x_{0}+ax_{1}=0$ in the Coble fourfold $Z$. If $a\neq 0$ then we have a $G$-invariant set which consists of three singular points $$p_{1}=(0:0:0:1:1:-1:-1),\ p_{2}=(0:0:0:1:-1:1:-1),\ p_{3}=(0:0:0:1:-1:-1:1)$$ in coordinates $(y:x_{1}:x_{2}:x_{3}:x_{4}:x_{5}:x_{6})$. Denote by $\pi:X\to \PP^{3}$ the morphism given by the linear system $|-\frac{1}{2}K_{X}|$. Let $l_{ij}$ be a line passing through $\pi(p_{i})$ and $\pi(p_{j})$. The preimage of $l_{ij}$ under $\pi$ consists of two curves which we denote by $l_{ij}'$ and $l_{ij}''$. Let $\widetilde{X}$ be a blow up of three singular points $p_{i}$. One can easily check that the divisor $-K_{\widetilde{X}}$ is nef and the only curves with zero intersection with it are six proper transforms of curves $l_{ij}'$ and $l_{ij}''$. Let $\widehat{X}$ be a variety which we obtain after making flops in such curves. The $G$-invariant Mori cone of the variety $\widehat{X}$ generated by two rays and one of them is $K_{\widehat{X}}$-negative. We need to prove that its contraction is not a divisorial contraction to $X$.

Suppose that we have the following commutative $G$-equivariant diagram:

$$\xymatrix{
 &  \widetilde{X}\ar@{-->}[r]\ar[ld]_{\tau} & \widehat{X}\ar[rd]^{\xi} &
\\
X & & & X
}$$
 where $\tau$ is a blow up of points $p_{1}, p_{2}$ and $p_{3}$ and $\xi$ is a contraction of a negative extremal ray. Let $H$ be an ample generator of the group $\Pic(X)^{G}$, let $\widetilde{H}=\tau^{*}H$ and let $\widetilde{E}$ be an exceptional divisor of $\tau$. Analogously we can define $\widehat{H}=\xi^{*}H$ and let $\widehat{E}$ be an exceptional divisor of $\xi$. Let $\widehat{H}'$ and $\widehat{E}'$ be the proper transforms of $\widehat{H}$ and $\widehat{E}$ on $\widetilde{X}$ respectively. The group $\Pic(\widetilde{X})^{G}$ is generated by $\widetilde{H}$ and $\widetilde{E}$ and we have following equalities:

$$\widetilde{H}^{3}=2,\ \widetilde{H}^{2}\cdot \widetilde{E}=\widetilde{H}\cdot \widetilde{E}^{2}=0,\ \widetilde{E}^{3}=6.$$

 Let $\widehat{H}'=a\widetilde{H}+b\widetilde{E}$ and $\widehat{E}'=c\widetilde{H}+d\widetilde{E}$. From equalities $$2\widetilde{H}-\widetilde{E}\sim -K_{\widetilde{X}}\sim 2\widehat{H}'-\widehat{E}'$$ we deduce that $c=2a-2$ and $d=1+2b$. Since the classes $\widehat{H}'$ and $\widehat{E}'$ generate the group~$\Pic(\widetilde{X})^{G}$ too, the determinant of the corresponding matrix is equal to $\pm 1$, so we have an equality $a+2b=\pm 1$. Also one can easily see that $$2=\widehat{H}^{3}=\frac{1}{2}\widehat{H}^{2}\cdot (-K_{\widehat{X}})=\frac{1}{2}(\widehat{H}')^{2}\cdot (-K_{\widetilde{X}})=2a^2-3b^2,$$ where the third equality follows from the fact that the anticanonical class $-K_{\widetilde{X}}$ is base point free, so we can assume that it does not intersect with flopped curves. The only solution of the system of equations $2=2a^2-3b^2$ and $a+2b=\pm 1$ in integers is $a=1, b=0$, but this contradicts our assumptions. Hence, the Sarkisov link which starts with the blow up of $X$ in three points $p_{i}$ gives us a birational transform of $X$ to another $G$-Mori fiber space, thus the variety $X$ is not $G$-birationally rigid.
\end{proof}
\begin{definition} In the notation of Lemma~\ref{le1} let $a=0$. A subgroup $S_{5}\subset \Aut(X)$ is a \emph{twisted subgroup}, if every permutation $\sigma\in S_{5}$ acts as $$(y:x_{1}:x_{2}:x_{3}:x_{4}:x_{5}:x_{6})\mapsto (\operatorname{sign}(\sigma)y:x_{\sigma(1)}:x_{\sigma(2)}:x_{\sigma(3)}:x_{\sigma(4)}:x_{\sigma(5)}:x_{6}).$$
\end{definition}
\begin{propos}\label{pr57} In the notation of Lemma~\ref{le1} let $a=0$. Let $G$ be a subgroup of $$\Aut(X)\simeq S_{5}\times C_{2}$$ such that $X$ is a $G\QQ$-factorial and $G$-minimal variety. Then the variety $X$ is $G$-birationally superrigid if and only if $G$ coincides with one of the following groups: $\Aut(X)$, twisted subgroup~$S_{5}$ or $A_{5}\times C_{2}$. Moreover, in this case the variety $X$ is $G$-birationally superrigid.
\end{propos}
\begin{proof} If the group $G$ is contained in $S_{4}\times C_{2}$ then the variety $X$ is not $G$-birationally rigid since we have the same link as in Proposition~\ref{pr58} and the proof works in our situations without changes. If the group $G$ is contained in $(C_{5}\rtimes C_{4})\times C_{2}$ then we have a $G$-invariant set of singular points of $X$. One can easily check that their images on $\PP^{3}$ are in general position and twisted cubics passing through them give us a structure of $G$-conic bundle on $X$. Indeed, for a general point of $\PP^{3}$ we have exactly one twisted cubic as above passing through this point and general twisted cubic intersects the variety $X$ in 5 double points and two additional points, so its preimage on $X$ is an irreducible curve of genus 0. Also we know, that the group $G$ cannot be a subgroup of $D_{12}\times C_{2}$, because such groups have only one- and two-dimensional irreducible representations while the representation $W$ of the group $G$ is a direct sum either of a one-dimensional and a four-dimensional irreducible representations or of two one-dimensional and a three-dimensional irreducible representations (see the proof of Lemma~\ref{le1}). Hence $G$ must be one of the following groups: $A_{5}$, $A_{5}\times C_{2}$, $S_{5}$ (two non-conjugate subgroups) or $S_{5}\times C_{2}$.

Due to~\cite[Corollary 8.2.2]{Pro1} there is a natural embedding of the group $\Aut(X)$ into the automorphism group of the root system $D_{5}$ and only $G$-invariant vector in the corresponding lattice is the null vector (otherwise variety $X$ is not $G\QQ$-factorial or $G$-minimal). The automorphism group of the root system $D_{5}$ is isomorphic to $C_{2}^{5}\rtimes S_{5}$ and acts on the corresponding lattice $\mathbb{Z}^{5}$ by changes of signs and permutations of coordinates. The group $A_{5}$ and standard subgroup $S_{5}\subset \Aut(X)$ have non-trivial invariant vector in the lattice, so $G$ can not coincide with them. One can easily see, that all other subgroups satisfy the required property, and equivariant birational rigidity with respect to them were proven in~\cite[Theorem 4.2]{CPS1}. Note, that there is a mistake in this paper, more precisely, the group $A_{5}$ is not minimal, as we saw before. The pair $(X, \frac{1}{\mu}\mathcal{M})$ is canonical for every $\mu$ and every movable $G$-invariant linear subsystem~\mbox{$\mathcal{M}\subset |-\mu K_{X}|$} (see the proof of the~\cite[Theorem 4.2]{CPS1}). So the variety $X$ is $G$-birationally superrigid by the Noether--Fano inequalities (see, for example,~\cite[Theorem 3.2.6]{ChS}).
\end{proof}

The proof of Theorem~\ref{th2} is a direct consequence of Lemma~\ref{le1}, Proposition~\ref{pr58} and Proposition~\ref{pr57}.

\end{document}